\documentclass[letterpaper, reqno,11pt]{article}
\usepackage[margin=1.0in]{geometry}
\usepackage{color,latexsym,amsmath,amssymb,path, amsthm}

\newtheorem{theorem}{Theorem}[section]
\newtheorem{conjecture}{Conjecture}[section]
\newtheorem{lemma}{Lemma}[section]
\newtheorem{corollary}{Corollary}[section]

\theoremstyle{remark}

\newtheorem{rem}{Remark}

\newcommand{\RR}{\mathbb{R}}

\newcommand{\CC}{\mathbb{C}}
\newcommand{\BZ}{\mathbf{Z}}

\newcommand{\pts}{\mathcal P}
\newcommand{\lines}{\mathcal L}

\newcommand{\sing}{\operatorname{sing}}
\newcommand{\reg}{\operatorname{reg}}

\newcommand{\itemizeEqnVSpacing}{\rule{0pt}{1pt}\vspace*{-14pt}}
\begin{document}
\pagenumbering{arabic}

\title{A note on rich lines in truly high dimensional sets}
\author{
Joshua Zahl\thanks{Massachusetts Institute of Technology, Cambridge, MA,  {\sl jzahl@mit.edu}.}
}

\date{\today}

\maketitle
\begin{abstract}
We modify an argument of Hablicsek and Scherr to show that if a collection of points in $\mathbb{C}^d$ spans many $r$--rich lines, then many of these lines must lie in a common $(d-1)$--flat. This is closely related to a previous result of Dvir and Gopi.
\end{abstract}

\renewcommand{\thefootnote}{\fnsymbol{footnote}} 
\renewcommand{\thefootnote}{\arabic{footnote}} 

\section{Introduction}
This note shows that the techniques of Hablicsek and Scherr from \cite{HS} can be extended from $\RR^d$ to $\CC^d$, with an $\epsilon$ loss in the exponent. In \cite{DG15}, Dvir and Gopi proved a new upper bound on the number of $r$--rich lines in ``truly'' $d$--dimensional configurations of points in $\CC^d$. Given a collection of points in $\CC^d$, Dvir and Gopi proved that either many of these points lie on a $(d-1)$--flat (i.e.~a $(d-1)$--dimensional affine subspace), or the points span few $r$--rich lines. Specifically, they established the following result.
 
\begin{theorem}[Dvir, Gopi]\label{DGTheorem}
For all $d \geq 1$, there exist constants $c_d,C_d$ such that the following holds. Let $\pts\subset\CC^d$ be a set of $n$ points, let $r\geq 2$ be an integer, and let $\mathcal{L}_r(\pts)$ be the set of lines that are incident to at least $r$ points from $\pts$. Suppose that for some $\alpha\geq 1$,
 \begin{equation*}
 |\mathcal{L}_r(\pts)|\geq C_d\cdot\alpha\cdot\frac{n^2}{r^d}.
 \end{equation*}
 Then there exists a subset $\pts^\prime\subset \pts$ of size at least $c_d \cdot\alpha\cdot \frac{n}{r^{d-2  }}$  contained in a $(d-1)$--flat.
 \end{theorem}
The bounds in Theorem \ref{DGTheorem} are not believed to be tight. Dvir and Gopi conjectured the following bound, which (if correct) would be tight.
 
 \begin{conjecture}[Dvir, Gopi]\label{DGConj}
 For $r\geq 2$, suppose $\pts\subset \CC^d$ is a set of $n$ points with
 \begin{equation*}
 |\mathcal{L}_r(\pts)|>\!\!\!>_d \frac{n^2}{r^{d+1}}+\frac{n}{r}.
 \end{equation*}
 Then there exists an integer $1<t<d$ and a subset $\pts^\prime\subset \pts$ of size $\gtrsim_{d} n/r^{d-t}$ contained in a $t$--flat.
 \end{conjecture}
This bound is a $d$--dimensional generalization of the (complex) Szemer\'edi-Trotter theorem \cite{Szemeredi, Toth,Zahl}. In \cite{HS}, Hablicsek and Scherr proved a stronger version of Theorem \ref{DGTheorem}, except Hablicsek and Scherr needed to replace complex lines in $\CC^d$ with real lines in $\RR^d$. This allowed them to use the discrete polynomial partitioning theorem \cite[Theorem 4.1]{Guth} and the joints theorem \cite{Guth2} of Guth and Katz. 
\begin{theorem}[Hablicsek, Scherr]\label{HSTHm}
For all $d \geq 1$, there exist constants $c_d,C_d$ such that the following holds. Let $\pts\subset\RR^d$ be a set of $n$ points, let $r\geq 2$ be an integer, and let $\mathcal{L}_r(\pts)$ be the set of lines that are incident to at least $r$ points from $\pts$. Suppose that
 \begin{equation*}
 |\mathcal{L}_r(\pts)|\geq C_d\cdot\frac{n^2}{r^{d+1}}.
 \end{equation*}
 Then there exists a subset $\pts^\prime\subset \pts$ of size at least $c_d \cdot \frac{n}{r^{d-1  }}$  contained in a $(d-1)$--flat.
 \end{theorem}

In this paper we will prove a theorem similar to Theorem \ref{HSTHm} in the original setting of Dvir and Gopi (i.e.~for complex lines in $\CC^d$). As in the Hablicsek and Scherr proof, the discrete polynomial partitioning theorem will play a major role. The joints theorem will not be used directly, but similar types of arguments will be employed.

\subsection{Incidence theorems over the complex numbers}
Since its introduction in 2010, the discrete polynomial partitioning theorem has been used to prove many incidence bounds in $\RR^d$. The theorem makes crucial use of the fact that removing a point from $\RR$ disconnects the real line into two components. This property is not true for $\CC$, which means that the discrete polynomial partitioning theorem cannot be employed directly to prove incidence theorems in $\CC^d$. Of course, one can identify $\CC$ with $\RR^2$ and then the discrete polynomial partitioning theorem becomes available again. Unfortunately, moving from $\CC^d$ to $\RR^{2d}$ often makes the problem appear more complicated, since the dimension of the objects being studied has now doubled. However, since the problem originally arose from a configuration in $\CC^d$, the new configuration in $\RR^{2d}$ may have some special properties that can aid in the analysis of the problem. This strategy was employed by T\'oth in \cite{Toth} to prove the Szemer\'edi-Trotter theorem in $\CC^2$ and by Sheffer and the author in \cite{SZ} to obtain an incidence theorem for points and curves in $\CC^2$. In the present work we will note several elementary relationships between collections of lines in $\CC^d$ and the corresponding two-flats in $\RR^{2d},$ and these observations will allow us to transfer the  Hablicsek-Scherr argument from $\RR^d$ to $\CC^d$. 
 
\subsection{Statement of the theorem}
\begin{theorem}\label{mainThm}
For all $d \geq 1$ and $\epsilon>0$, there exist constants $c_{d,\epsilon},C_{d,\epsilon}$ such that the following holds. Let $\pts\subset\CC^d$ be a set of $n$ points and let $r\geq 2$ be an integer. Suppose that for some $\alpha\geq 1$,
\begin{equation}\label{manyRRichLines}
|\mathcal{L}_r(\pts)|> C_{d,\epsilon} \cdot\alpha\cdot \frac{n^{2+\epsilon}}{r^{d+1}}.
\end{equation}
Then there exists a subset $\pts^\prime\subset \pts$ of size at least $c_{d,\epsilon}\cdot\alpha\cdot \frac{n^{1+\epsilon}}{r^{d-1}}$ contained in a $(d-1)$--flat.
\end{theorem}
Theorem \ref{mainThm} is not strictly stronger than Theorem \ref{DGTheorem} because Theorem \ref{mainThm} contains the term $n^{2+\epsilon}$ rather than $n^2$. However, if $r$ is not too small compared to $n$, then one can ``trade'' the term $n^{\epsilon}$ for a term of the form $r^{\epsilon_1}$. More precisely, we have the following. 
\begin{corollary}[Cheap Dvir-Gopi]
For all $d \geq 1$ and $\epsilon_0>0$, there exist constants $c_{d,\epsilon_0},C_{d,\epsilon_0}$ such that the following holds. Let $\pts\subset\CC^d$ be a set of $n$ points and let $r\geq n^{\epsilon_0}$ be an integer. Suppose that for some $\alpha^\prime\geq 1$,
\begin{equation}\label{manyRRichLinesCheap}
|\mathcal{L}_r(\pts)|> C_{d,\epsilon_0} \cdot\alpha^\prime\cdot \frac{n^{2}}{r^{d}}.
\end{equation}
Then there exists a subset $\pts^\prime\subset \pts$ of size at least $c_{d,\epsilon_0}\cdot\alpha^\prime\cdot \frac{n}{r^{d-2}}$ contained in a $(d-1)$--flat.
\end{corollary}
\begin{proof}
Let $\rho=\log r/\log n$; by assumption $\rho\geq\epsilon_0$, and for each $\epsilon>0,\ n^\epsilon = r^{\epsilon/\rho}.$ Select $\epsilon=\rho/2$, and let $\alpha = \alpha^\prime r^{1/2}$. Applying Theorem \ref{mainThm}, we conclude that either
\begin{equation}
\begin{split}
|\mathcal{L}_r(\pts)|&> C_{d,\epsilon} \cdot\alpha\cdot \frac{n^{2+\epsilon}}{r^{d+1}}\\
&> C_{d,\epsilon_0} \cdot\alpha^\prime \cdot\frac{n^{2}}{r^{d}},
\end{split}
\end{equation}
or there exists a subset $\pts^\prime\subset \pts$ of size at least $c_{d,\epsilon}\cdot\alpha\cdot \frac{n^{1+\epsilon}}{r^{d-1}}=c_{d,\epsilon}\cdot\alpha^\prime \cdot \frac{n}{r^{d-2}}$ contained in a $(d-1)$--flat.
\end{proof}

\subsection{Initial reductions}
In this section we will show that in order to prove Theorem \ref{mainThm} it suffices to prove the following lemma:
\begin{lemma}\label{mainLem}
For all $d \geq 1$ and $\epsilon>0$, there exist constants $c^\prime_{d,\epsilon},C^\prime_{d,\epsilon}$ such that the following holds. Let $\pts_1\subset\CC^d$ be a set of $n_1$ points and let $\mathcal{L}_1$ be a set of $\ell_1$ lines in $\CC^d$. Then at least one of the following statements must hold:
\begin{enumerate}

\item[(A)]  There is a point $p\in\pts_1$ and a complex $(d-1)$--flat $\Pi\subset\CC^d$ so that
\begin{equation}\label{onePopPt}
|\{L\in \lines_1\colon p\in L,\ L\subset\Pi\}|\geq c^\prime_{d,\epsilon}|\{L\in \lines_1\colon p\in L\}|.
\end{equation}
\item[(B)]\itemizeEqnVSpacing 
\begin{equation}\label{notTooManyIncidences}
I(\pts_1,\lines_1)\leq C^\prime_{d,\epsilon}\big(n_1^{\frac{2+\epsilon}{d+1}}\ell_1^{\frac{d}{d+1}}+n_1+\ell_1\big).
\end{equation}
\end{enumerate}
\end{lemma}

To obtain Theorem  \ref{mainThm} from Lemma \ref{mainLem}, fix $\epsilon>0$. Let $\pts\subset\CC^d$ be a set of $n$ points and suppose \eqref{manyRRichLines} holds. If $r\geq \alpha^{1/d}n^{\frac{1+\epsilon}{d}}$, then 
\begin{equation*}
\alpha\frac{n^{1+\epsilon}}{r^{d-1}}\leq r,
\end{equation*}
so Theorem  \ref{mainThm} holds with $c_{d,\epsilon}=1$; just select any line from $\mathcal{L}_r(\pts)$, let $\pts^\prime$ be the set of points incident to the line, and select any $(d-1)$--flat containing the line.

Henceforth we will assume that $r< \alpha^{1/d}n^{\frac{1+\epsilon}{d}}$. Repeating an argument of Hablicsek and Scherr \cite{HS}, we can find a set $\lines_1\subset\mathcal{L}_{r}(\pts)$ and a set  $\pts_1\subset\pts$ so that each line in $\lines_1$ is incident to at least $r/4$ points from $\pts_1$, each point in $\pts_1$ is incident to at least $\frac{1}{4}C^\prime_{d,\epsilon} \cdot\alpha\cdot \frac{n^{1+\epsilon}}{r^{d}}$ lines from $\lines_1$, and $I(\pts_1,\lines_1)\geq\frac{1}{2}I(\pts,\mathcal{L}_r(\pts))$. In brief, the argument is as follows. If $G=(A\sqcup B,E)$ is a bipartite graph, then we can find subsets $A^\prime\subset A,\ B^\prime\subset B$ so that in the induced subgraph, each vertex from $A$ has degree at least $\frac{|E|}{4|A|}$, each vertex from $B$ has degree at least  $\frac{|E|}{4|B|}$, and $|E^\prime|\geq|E|/2$. We apply this lemma to the bipartite graph with edge set $A=\pts,\ B=\mathcal{L}_{r}(\pts),$ and where $a\sim b$ if the point corresponding to $a$ lies on the line corresponding to $b$. See \cite{HS} for details.

Let $p\in\pts_1$ and let $\rho>0$. Observe that if a $\rho$--fraction of the lines from $\lines_1$ passing through $p$ lie in a common $(d-1)$--flat, then at least $\frac{\rho}{16}C_{d,\epsilon} \cdot\alpha\cdot \frac{n^{1+\epsilon}}{r^{d-1}}$ points from $\pts_1$ lie in a common $(d-1)$--flat. 

Apply Lemma \ref{mainLem} to the arrangement $(\pts_1,\lines_1)$ (with the same value of $\epsilon$). If condition (A) from Lemma \ref{mainLem} holds, then there exists a subset of $\pts^\prime\subset\pts$ of cardinality at least $\frac{1}{16}c^\prime_{d,\epsilon}\cdot \alpha\cdot \frac{n^{1+\epsilon}}{r^{d-1}}$ contained in a $(d-1)$--flat.

If condition (A) does not hold, then condition (B) must hold, and this implies that 
$$I(\pts,\mathcal{L}_{r}(\pts))\leq C^\prime_{d,\epsilon}\big(n_1^{\frac{2+\epsilon}{d+1}}\ell_1^{\frac{d}{d+1}}+n_1+\ell_1\big).$$ 

But since each line from $\mathcal{L}_{r}(\pts)$ is $r$--rich, we conclude that 
\begin{equation*}
\begin{split}
r|\mathcal{L}_{r}(\pts)|&\leq 2C^\prime_{d,\epsilon}\big(n_1^{\frac{2+\epsilon}{d+1}}\ell_1^{\frac{d}{d+1}}+n_1+\ell_1\big)\\
&\leq 2C^\prime_{d,\epsilon}\big(n^{\frac{2+\epsilon}{d+1}}|\mathcal{L}_{r}|^{\frac{d}{d+1}}+n+|\mathcal{L}_{r}|\big),
\end{split}
\end{equation*}
and thus
\begin{equation}
\begin{split}
|\mathcal{L}_{r}|&\leq 2^{d+1}C^\prime_{d,\epsilon}\big(\frac{n^{2+\epsilon}}{r^{d+1}}+\frac{n}{r}\big)\\
&\leq 2^{d+1}C^\prime_{d,\epsilon}\alpha \frac{n^{2+\epsilon}}{r^{d+1}},
\end{split}
\end{equation}
where on the second line we used the fact that $\alpha\geq 1$ and $r< \alpha^{1/d}n^{\frac{1+\epsilon}{d}}$. This contradicts the assumption that \eqref{manyRRichLines} holds. Thus we obtain Theorem \ref{mainThm} with $c_{d,\epsilon}=\frac{1}{16}c^\prime_{d,\epsilon}$ and $C_{d,\epsilon}=2^{d+1}C^\prime_{d,\epsilon}$. 

\subsection{Main proof ideas}
The rest of this paper will be devoted to proving Lemma \ref{mainLem}. The basic idea is to use a bounded-degree partitioning polynomial and prove the lemma by induction on $n_1+\ell_1$. Theorem \ref{mainThm} does not survive the process of induction, but Lemma \ref{mainLem} does---this is why we prove Lemma \ref{mainLem} first rather than proving Theorem \ref{mainThm} directly.

Here are the main steps. We regard complex lines in $\CC^d$ as two-flats in $\RR^{2d}$. We will then prove Lemma \ref{mainLem} by induction on the number of points and flats. We partition $\RR^{2d}$ into cells using a bounded-degree partitioning polynomial; inside each cell we can apply the induction hypothesis. Either there is a point inside a cell satisfying condition (A) from the lemma, or the total number of incidences inside the cells is controlled by \eqref{notTooManyIncidences}. 

We must now deal with incidences occurring on the boundary of the partition. If $p$ is a point lying on the boundary of the partition and $X$ is a two-flat (arising from a complex line) that is incident to $p$, then $X$ is either contained in the boundary of the partition or intersects the boundary in a bounded-degree algebraic set of dimension at most one. 

If the former option occurs most of the time (for a given point $p$), then there must exist a $(2d-1)$--flat in $\RR^{2d}$ containing many two-flats, each of which contains $p$, and these two-flats in turn arise from complex lines. This implies that many complex lines are contained in a complex $(d-1)$--flat in $\CC^d$. 

If the latter option occurs most of the time, then after applying a generic linear transformation we are reduced to a problem that is very similar to our original one, except now we are dealing with points and bounded-degree curves in $\RR^d$, rather than points and two-flats in $\RR^{2d}$. A similar argument shows that either the number of incidences is controlled by \eqref{notTooManyIncidences}, or there is a point $p$ and a $(d-1)$--flat in $\RR^{2d}$ containing $p$ so that many curves passing through $p$ are tangent to this $(d-1)$--flat at the point $p$. This implies that in the original configuration of complex lines there are many lines passing through a common point that are contained in a complex $(d-1)$--flat. 
\subsection{Thanks}

The author would like to thank M\'arton Hablicsek for pointing out errors in an earlier version of this manuscript and to the two anonymous referees for their helpful suggestions. The author was supported by a NSF Postdoctoral Fellowship.

\section{Preliminaries}
\subsection{Real and complex vectors}
Let $\iota\colon\CC^d\to\RR^{2d}$ be the map $\iota(x_1+iy_1,\ldots,x_d+iy_d)\mapsto(x_1,y_1,\ldots,x_d,y_d).$ If $v=(a_1,b_1,\ldots,a_d,b_d)\in\RR^{2d}$, then $\iota^{-1}(v) = (a_1+ib_1,\ldots,a_d+ib_d)\in\CC^d$.

If $X_1,\ldots,X_k$ are vector-spaces in $\RR^d$ (resp. $\CC^d$), let $\operatorname{span}_{\RR}\{X_1,\ldots,X_k\}$ (resp. $\operatorname{span}_{\CC}\{X_1,\ldots,X_k\}$) be the linear span of these vector spaces, regarded as subspaces of $\RR^d$ or $\CC^d$. Abusing notation, we will identify the non-zero vector $v\in\RR^d$ with the one-dimensional vector space $\RR v$, and similarly for vectors $v\in\CC^d$.

If $v=(a_1,b_1,\ldots,a_d,b_d)\in\RR^{2d},$ define $v^{\dagger}=\operatorname{span}_{\RR}\{(a_1,b_1,\ldots,a_d,b_d),(-b_1,a_1,\ldots,-b_d,a_d)\}$. By definition, if $v\in\RR^{2d}$ then
$$
v^{\dagger}=\iota(\operatorname{span}_{\CC}(\iota^{-1}(v))).
$$
I.e.,~if $v\in\RR^{2d}$ is a vector, let $\iota^{-1}(v)\in\CC^d$ be its pre-image under $\iota$. Then the span of $\iota^{-1}(v)$ is a one-dimensional complex subspace of $\CC^d$. The image of this subspace under $\iota$ is a two dimensional real subspace of $\RR^{2d}$. This is precisely $v^{\dagger}$.

Finally, if $\Pi\subset\RR^{2d}$ is a vector space, define $\Pi^\dagger=\operatorname{span}_{\RR}\{v^{\dagger}\colon v\in\Pi\}$. Note that $\dim(\Pi^{\dagger})\leq 2\dim(\Pi)$.

\subsection{Linear Algebra}

\begin{lemma}\label{vectorsTrappedInRealSpaceImpliesTrappedInComplex}
Let $L_1,\ldots,L_k$ be lines in $\CC^d$ passing through the origin. If there is a $(2k-1)$--flat $\Pi_0\subset\RR^{2d}$ containing $\iota(L_1),\ldots,\iota(L_k)$, then there is a $(d-1)$--flat $\Pi\subset\CC^d$ containing $L_1,\ldots,L_d$. 
\end{lemma}
\begin{proof}
Since $L_1,\ldots,L_k$ pass through the origin, $\Pi_0$ must also contain the origin. Let $v_0\in\RR^{2d}$ be a vector orthogonal to $\Pi_0$. Then $v_0$ is orthogonal to $\iota(L_j)$ for each $j=1,\ldots,k$. This implies that $\iota^{-1}(v_0)$ is orthogonal to $L_j$ for each index $j=1,\ldots,k$. Let $\Pi$ be the orthogonal compliment of $\iota^{-1}(v_0)$.
\end{proof}

\begin{lemma}\label{dependentRealVectorsImpliesDependentComplexVectors}
Let $L_1,\ldots,L_k$ be lines in $\CC^d$ passing through the origin. For each index $i$, let $v_i\in\iota(L_i)$ be a non-zero vector. Let $\Pi_0\subset\RR^{2d}$ be a $(d-1)$--dimensional vector space. If $v_1,\ldots,v_k\in\Pi_0$, then there is a $(d-1)$--dimensional subspace $\Pi\subset\CC^d$ containing $L_1,\ldots,L_k$.
\end{lemma}
\begin{proof}
Let $w_1,\ldots,w_{d-1}\in\RR^{2d}$ be vectors that span $\Pi_0$. For each index $j$ we have $v_j = \sum_i a_{ij}w_i$, and thus $\iota^{-1}(v_j) = \sum_i \iota^{-1}(a_{ij}w_i)$ But this implies that $L=\CC \iota^{-1}(v_j) \subset  \operatorname{span}_{\CC}\{\iota^{-1}(w_1),\ldots,\iota^{-1}(w_d)\}$. Define $\Pi =\operatorname{span}_{\CC}\{\iota^{-1}(w_1),\ldots,\iota^{-1}(w_d)\}.$
\end{proof}

\begin{lemma}\label{dMOneRealDimInComplex}
Let $v_1,\ldots,v_k\in\RR^{2d}$. Let $\Pi\subset\RR^{2d}$ be a $(d-1)$--dimensional vector space, and suppose $v_1,\ldots,v_k\in \Pi$. Then $v_1^\dagger,\ldots,v_k^\dagger \subset \Pi^\dagger$.
\end{lemma}

\subsection{Real and complex varieties}\label{realCplxVarSec}
We will need several properties of real and complex affine varieties. For complex varieties, a good introduction can be found in  \cite{Harris}, while for real varieties \cite{BCR} is a good source. 

Let $K=\RR$ or $\CC$. Let $V\subset K^d$ be an irreducible variety, and let $I(K)$ be the ideal of polynomials in $K[x_1,\ldots,x_d]$ that vanish on $V$. For $x\in K$, we define the Zariski tangent space of $V$ at $x$ to be
\begin{equation*}
T_xV = \{v \in K^d \colon v\cdot \nabla f(x)=0\ \textrm{for all}\ f\in I(V)  \}.
\end{equation*} 
 
We always have $\dim(T_xV)\geq \dim(V)$ (in the case $K=\RR$, the dimension of a real variety is slightly subtle; see \cite{BCR} for details). If $\dim(T_xV)\geq \dim(V)$, then $x$ is a singular point of $V$. Let $V_{\sing}$ be the set of singular points of $V$, and let $V_{\reg}=V\backslash V_{\sing}$ be the set of regular points of $V$. See  \cite{Harris} and \cite{BCR} for further background and details.

If $V\subset\CC^d$ is a complex variety, let $V(\RR)\subset\RR^d$ be its real locus. In particular, if $V\subset\CC^d$ is a one-dimensional variety, then $V(\RR)$ is a real variety of dimension at most 1. If $V\subset\RR^d$ is a real variety, let $V^*\subset\CC^d$ be the complexification of $V$---this is the smallest complex variety whose real locus contains $V$. If $p\in\RR^d$ is a point, then $p^*\subset\CC^d$ is the image of $p$ under the usual embedding from $\RR^d\to\CC^d$.

Most of our arguments will occur over the reals. However, we will sometimes need to work over $\CC$ in order to make use of the following result of Solymosi and Tao:
\begin{lemma}[Solymosi-Tao \cite{ST}, Corollary 4.5]\label{CoverVarietySmoothPts}
Let $V\subset\CC^d$ be a $k$--dimensional algebraic variety of degree at most $D$. Then one can cover $V$ by $V_{\reg}$ and $O_{D,d}(1)$ sets of the form $W_{\reg}$, where each $W$ is an algebraic variety contained in $V$ of dimension at most $k-1$ and degree $O_{D,d}(1)$.
\end{lemma}

Finally, if $P\in\RR[x_1,\ldots,x_d]$ is a polynomial, define $\BZ(P)=\{x\in\RR^d\colon P(x)=0\}$. 
\subsection{The discrete polynomial partitioning theorem}
We will make crucial use of the discrete polynomial partitioning theorem of Guth and Katz:
\begin{theorem}[Guth-Katz, \cite{Guth}, Theorem 4.1]\label{th:partition}
Let $\pts$ be a set of $n$ points in $\RR^d$. For each $D\geq 1$, there exists a polynomial $P$ of degree $\leq D$ so that $\RR^{d}\backslash\BZ(P)$ is a union of $O_{d}(D^d)$ connected components (cells), and each cell contains $O_d(nD^{-d})$ points of $\pts$.
\end{theorem}

\subsection{Connected components of real varieties}
The discrete polynomial partitioning theorem will be used to partition the point-line arrangement (technically, the point-two--flat arrangement) into connected components, which are called cells. Each point lies in at most one cell. While each two--flat can enter several cells, the following theorem controls how many cells a given two--flat can enter.
\begin{theorem}[Barone-Basu \cite{BB}, special case]\label{BBThm}
Let $Y\subset\RR^d$ be a real variety of dimension $e$ that can be defined by polynomials of degree at most $D_1$, and let $Z\subset\RR^d$ be a real variety that can be defined by polynomials of degree at most $D_2$. Then $Y\backslash Z$ contains $O_{d,D_1}(D_2^e)$ connected components. 
\end{theorem}
\begin{rem}
In the special case where $Y\subset\RR^d$ is a two--flat, the above result also follows from the Milnor-Thom theorem. However, we will also be interested in the case where $Y$ is defined by polynomials of larger (though still bounded) degree. 
\end{rem}

\section{Proof of Lemma \ref{mainLem}, Step 1: Lines in $\CC^d$}
\begin{proof}[Proof of Lemma \ref{mainLem}]
We prove the lemma by induction on $n_1+\ell_1$. First, note that $I(\pts_1,\lines_1)\leq n_1^2+\ell_1$. Thus if we select $C^{\prime}_{d,\epsilon}$ sufficiently large, we can assume that $\ell_1<\rho n_1^{\frac{2}{d+1}}\ell_1^{\frac{d}{d+1}}$; we can make $\rho>0$ arbitrarily small by making the constant $C^{\prime}_{d,\epsilon}$ larger.

Let $P$ be a partitioning polynomial in $\RR^{2d}$ of degree $D$ adapted to the set $\iota(\pts_1)\subset\RR^{2d}$, as given by Theorem \ref{th:partition}; there are $O_d(D^{2d})$ cells. We will choose $D$ later, and it will depend only on $d$ and $\epsilon$. For each cell $\Omega,$ let $\ell_\Omega$ be the number of lines $L\in\lines_1$ so that $\iota(L)\cap\Omega\neq\emptyset$. By Theorem \ref{BBThm}, we have
\begin{equation}
\sum_{\Omega}\ell_\Omega = O_d(D^2\ell_1).
\end{equation}

Apply the induction hypothesis inside each cell. Either property (A) holds inside some cell, or we have
\begin{equation}\label{inductionCalc}
\begin{split}
I(\iota(\pts_1)\backslash \BZ(P),\iota(\mathcal L_1)) &=\sum_\Omega I\big(\iota(\pts) \cap\Omega,\ \iota(\lines_1)\big)\\
&\leq \sum_{\Omega}  C^{\prime}_{d,\epsilon} \big( O_d(n_1D^{-2d})^{\frac{2+\epsilon}{d+1}}\ell_\Omega^{\frac{d}{d+1}}+|\iota(\pts_1)\cap\Omega|+\ell_\Omega\big)\\
&\leq C^{\prime}_{d,\epsilon} \big( O_d(1)D^{-\frac{2d\epsilon}{d+1}}n_1^{\frac{2+\epsilon}{d+1}}\ell_1^{\frac{d}{d+1}} + |\iota(\pts_1)\backslash Z(P)|+O_d(D^2\ell_1)\big)\\
&\leq C^{\prime}_{d,\epsilon} \big(\ (O_d(1)D^{-\frac{2d\epsilon}{d+1}})\ n_1^{\frac{2+\epsilon}{d+1}}\ell_1^{\frac{d}{d+1}} + |\iota(\pts_1)\backslash Z(P)|+\frac{1}{4}n_1^{\frac{2}{d+1}}\ell_1^{\frac{d}{d+1}}\big)\\
&\leq C^{\prime}_{d,\epsilon} \big(\frac{1}{2}n_1^{\frac{2+\epsilon}{d+1}}\ell_1^{\frac{d}{d+1}} + |\iota(\pts_1)\backslash Z(P)|\big).
\end{split}
\end{equation}
Here we selected $D$ sufficiently large (depending on $d$ and $\epsilon$) so that the term $O_d(1) D^{-\frac{2d\epsilon}{d+1}}$ is at most $1/4$, and we selected $C^{\prime}_{d,\epsilon}$ sufficiently large (depending on $D$, $d$, and $\epsilon$, which in turn depend only on $d$ and $\epsilon$) to guarantee that the term $O_d(D^2\ell_1)$ is at most $\frac{1}{4}n_1^{\frac{2}{d+1}}\ell_1^{\frac{d}{d+1}}$.

Applying Lemma \ref{CoverVarietySmoothPts} to $\BZ(P)^*$, we can find a collection $\mathcal{V}$ of $O_{d,D}(1)$ complex varieties in $\CC^{2d}$ and sets $\{\pts_V\}_{V\in\mathcal{V}}$ so that the following properties hold:
\begin{itemize}
 \item Each variety in $\mathcal{V}$ is of degree $O_{d,D}(1)$.
 \item For each $V\in\mathcal{V}$, $\pts_V\subset \iota(\pts)\cap V(\RR)$.
 \item The sets $\{\pts_V\}_{V\in\mathcal{V}}$ are disjoint, and $\bigcup_{V\in\mathcal{V}}\pts_V=\iota(\pts)\cap\BZ(P)$.
 \item For each $p\in\pts_V$, $p^*$ is a regular point of $V$.
\end{itemize}

For each $V\in\mathcal{V}$, let 
\begin{equation*}
\mathcal{L}_V=\{L\in\lines_1\colon(\iota(L))^*\subset V\}. 
\end{equation*}

One of the following must hold:
\begin{enumerate}
\item[(A.1)] There is a variety $V\in\mathcal{V}$ and a point $p\in\pts_V$ so that
\begin{equation}\label{manyFlatsInAVariety}
|\{L\in \mathcal{L}_V\colon p\in L\}|\geq\frac{1}{2}|\{L\in\lines_1\colon p\in L\}|.
\end{equation}
\item[(B.1)] \eqref{manyFlatsInAVariety} fails for every $V\in\mathcal{V}$ and every $p\in\pts_V$.
\end{enumerate}

Suppose (A.1) holds for some variety $V$ and some point $p\in\pts_V$. Let $\Pi_0\subset\RR^{2d}$ be a $(2d-1)$--flat containing $(T_{p^*}V)(\RR)$. If $\iota(L)^*\subset V$ and $p\in L$, Then $L\subset \Pi_0$. Thus $\iota(L)\subset \Pi_0$ for at least half of the complex lines in $\{L\in\mathcal{L}_1\colon p\in L\}$. By Lemma \ref{vectorsTrappedInRealSpaceImpliesTrappedInComplex}, there exists a complex $(d-1)$-flat $\Pi$ so that at least half the complex lines in $\{L\in \mathcal{L}_1\colon p\in L\}$ are contained in $\Pi$. We conclude that (A) holds, provided $c^\prime_{d,\epsilon}\geq 1/2$.

Now suppose (B.1) holds. We will consider each $V\in\mathcal{V}$ in turn. If (A) holds for some point $p\in \pts_V$ then we are done. If (A) fails for every point $p\in \pts_V$, then we will establish the bound
\begin{equation}\label{boundPtsVIncidences}
 I(\pts_V, \iota(\lines_1))\leq C_{d,\epsilon}^{\prime\prime}\big(n_1^{\frac{2+\epsilon}{d+1}}\ell_1^{\frac{d}{d+1}}+|\pts_V|+\ell_1\big).
\end{equation}
Since there are $O_{d,\epsilon}(1)$ elements of $\mathcal{V}$, if we establish \eqref{boundPtsVIncidences} for each $V\in\mathcal{V}$ then we will have proved the bound
\begin{equation}
I(\pts\cap Z(P), \iota(\lines_1))\leq C_{d,\epsilon}^{\prime}\big(\frac{1}{2}n_1^{\frac{2+\epsilon}{d+1}}\ell_1^{\frac{d}{d+1}}+|\pts\cap Z(P)|+\ell_1\big),
\end{equation}
provided $C_{d,\epsilon}^{\prime}$ is chosen sufficiently large (depending only on $d$ and $\epsilon$), and this will conclude the proof of Lemma \ref{mainLem}. It remains to prove \eqref{boundPtsVIncidences} for each $V\in\mathcal{V}$; recall that by assumption, (B.1) holds and (A) fails.

Since (A) fails, by Lemma \ref{dMOneRealDimInComplex}, for each $p\in\pts_V$ and each $(d-1)$--flat $\Pi\subset\RR^{2d}$ containing $v$, we have
\begin{equation*}
|\{L\in \mathcal{L}_1\colon p\in L,\ \dim(\Pi\cap \iota(L))\geq 1\}|\leq c^\prime_{d,\epsilon}|\{L\in \mathcal{L}_1\colon p\in L\}|.
\end{equation*}
Since $\dim\big(\Pi\cap \iota(L)\big)=\dim\big(\Pi^*\cap \iota(L)^*\big)$ (where the left $\dim$ is the real dimension of the real algebraic set, and the right $\dim$ is the complex dimension of the complex algebraic set), we have
\begin{equation*}
|\{L\in \mathcal{L}_1\colon p\in L,\ \dim\big(\Pi^*\cap \iota(L)^*\big)\geq 1\}|\leq c^\prime_{d,\epsilon}|\{L\in \mathcal{L}_1\colon p\in L\}|,
\end{equation*}
and thus
\begin{equation}\label{smallFractionOfTangenciesInDmOnePlane}
|\{L\in \mathcal{L}_1\colon p\in L,\ \dim\big(\Pi^*\cap T_p(\iota(L)^*\cap V)\big)\geq 1\}|\leq c^\prime_{d,\epsilon}|\{L\in \mathcal{L}_1\colon p\in L\}|,
\end{equation}
where $T_p(\iota(L)^*\cap V)\subset\CC^{2d}$ is the Zariski tangent space of the (complex) variety $\iota(L)^*\cap V$ at $p$.

Let $\pi\colon\RR^{2d}\to\RR^d$ be a generic (with respect to $\pts$, $\mathcal{L}$, and $V$) linear transformation. $\pi$ also extends to a map from $\CC^{2d}\to\CC^d$. Since the projection of a linear algebraic variety cannot increase its dimension (this is actually true for all algebraic varieties, though we don't need this fact), we have 
\begin{equation*}
|\{L\in \mathcal{L}_1\colon p\in L,\ \dim\big(\pi(\Pi^*)\cap \pi(T_p(\iota(L)^*\cap V))\big)\geq 1\}|\leq c^\prime_{d,\epsilon}|\{L\in \mathcal{L}_1\colon p\in L\}|
\end{equation*}
for every $(d-1)$--flat $\Pi\subset\RR^{2d}$ containing $p$. This implies that
\begin{equation}\label{notTooManyVectorsInLinSpace}
|\{L\in \mathcal{L}_1\colon p\in L,\ \dim\big(\Pi^*\cap \pi(T_p(\iota(L)^*\cap V))\big)\geq 1\}|\leq c^\prime_{d,\epsilon}|\{L\in \mathcal{L}_1\colon p\in L\}|
\end{equation}
for every $(d-1)$--flat $\Pi\subset\RR^d$ containing $p$.

Let $\Gamma_0$ be the set of irreducible components of $\pi\big(\iota(L)^*\cap V\big)$ as $L$ ranges over the set $\lines_1\backslash \mathcal{L}_V$. We have
\begin{equation}\label{bdOnGamma0}
|\Gamma_0|=O_{d,D}(\ell_1)=O_{d,\epsilon}(1).
\end{equation}
Furthermore, if $L\in\lines_1$, then $O_{d,\epsilon}(1)$ curves from $\Gamma_0$ can be contained in $\pi(\iota(L)^*)$ (this is because the projection $\pi$ was chosen generically with respect to $\lines$). Since there is at most one complex line from $\lines_1$ passing through any two points in $\CC^d,$ there are at most $O_{d,\epsilon}(1)$ curves from $\Gamma_0$ passing through any two points in $\CC^{2d}$.

Let $\pts^\prime=\pi(\iota(\pts_V))$; technically $\pts^\prime\subset\CC^d$, but the points in $\pts^\prime$ are real, i.e.~all coordinates are real. As noted above, the arrangement $(\pts^\prime,\Gamma_0)$ has 2 degrees of freedom and multiplicity type $O_{d,\epsilon}(1)$---this means that any two curves intersect in at most $O_{d,\epsilon}(1)$ points, and if we fix two points from $\pts^\prime$ then at most $O_{d,\epsilon}(1)$ curves pass through both of them.  

We will need the following lemma.

\begin{lemma}\label{curvesResult}
Let $\pts_2\subset\RR^d$ be a collection of $n_2$ points. Let $\Gamma_0$ be a collection of $\ell_2$ irreducible complex curves in $\CC^d$, each of degree at most $C$. Let $\Gamma=\{\alpha(\RR)\colon\alpha\in\Gamma_0\}$. Suppose that $(\pts_2,\Gamma)$ has two-degrees of freedom and multiplicity type $s$. Then at least one of the following must hold.
\begin{enumerate}
\item[(A.2)]  there is a point $p\in\pts_2$ and a $(d-1)$--flat $\Pi\subset\RR^d$ containing $p$ so that
\begin{equation}\label{manyCurveTangentToAPlane}
|\{\gamma\in\Gamma\colon p\in\gamma,\ \dim(T_p\gamma\cap\Pi)\geq 1\}|\geq\frac{1}{2}|\{\gamma\in\Gamma\colon p\in\gamma\}|,
\end{equation}
where $T_p\gamma$ is the (real) Zariski tangent space of $\gamma$ at $p$ (see Section \ref{realCplxVarSec} for a definition of the Zariski tangent space).
\item[(B.2)]\itemizeEqnVSpacing
\begin{equation}\label{IncidenceBdCurves}
I(\pts_2,\Gamma)\leq C^{\prime\prime}_{\epsilon}(n_2^{\frac{2+\epsilon}{d+1}}\ell_2^{\frac{d}{d+1}}+n_2+\ell_2),
\end{equation}
where the constant $C^{\prime\prime}_{\epsilon}$ depends on $\epsilon,$ $d$, $C,$ and $s$.
\end{enumerate}
\end{lemma}

To avoid breaking the flow of the argument, we will defer the proof of Lemma \ref{curvesResult} to the next section. 

For each $V\in\mathcal{V}$, apply Lemma \ref{curvesResult} to the collection $(\pts_V,\Gamma_0)$. If (B.2) holds, then
\begin{equation*}
I(\pts_V,\iota(\Gamma_0))\leq C^{\prime\prime}_\epsilon( |\pts_V|^{\frac{2+\epsilon}{d+1}}|\Gamma_0|^{\frac{d}{d+1}}+|\pts_V|+|\Gamma_0|).
\end{equation*}
Recall that the constant $C_{\epsilon}^{\prime\prime}$ depends only on $d,\epsilon,C,$ and $s$. $C$ and $s$ in turn depend on $D$, which depends only on $\epsilon$ and $d$. Thus $C_{\epsilon}^{\prime\prime}$ depends only on $\epsilon$ and $d$. By \eqref{bdOnGamma0}, $ |\Gamma_0| =O_{d,\epsilon}(\ell_1)$. We conclude that 
\begin{equation*}
I(\pts_V,\iota(\mathcal{L}_1\backslash\mathcal{L}_V))\leq C^{\prime\prime}_{d,\epsilon}( n_1^{\frac{2+\epsilon}{d+1}}\ell_1^{\frac{d}{d+1}}+|\pts_V|+\ell_1).
\end{equation*}
Finally, since (A.1) fails, we have 
\begin{equation*}
I(\pts_V,\iota(\mathcal{L}_1))  \leq 2 I(\pts_V,\iota(\mathcal{L}_1\backslash\mathcal{L}_V)),
\end{equation*}
so \eqref{boundPtsVIncidences} holds. Thus if (B.2) holds for each $V\in\mathcal{V}$ then we are done.

Now suppose (A.2) holds for some $V\in\mathcal{V}$ and some point $p_1\in\pts^\prime$, and let $\Pi\subset\RR^d$ be the corresponding $(d-1)$--flat. Then 
\begin{equation*}
|\{\gamma\in\Gamma\colon p\in\gamma,\ \dim(T_p\gamma^*\cap\Pi^*)\geq 1\}|\geq\frac{1}{2}|\{\gamma\in\Gamma\colon p\in\gamma\}|.
\end{equation*}
Since each $L\in\mathcal{L}$ contributes $O_{d,\epsilon}(1)$ curves to $\Gamma$, if we choose the constant $c^\prime_{d,\epsilon}$ sufficiently small (depending only on $d$ and $\epsilon$), then 
\begin{equation*}
|\{L\in\mathcal{L}\colon \dim(T_p \pi(L^*\cap V) \cap\Pi^*)\geq 1\}|\geq c^\prime_{d,\epsilon}|\{L\in\mathcal{L}\colon p\in L\}|,
\end{equation*}
which violates \eqref{notTooManyVectorsInLinSpace}. This concludes the proof of Lemma \ref{mainLem}, modulo the proof of Lemma \ref{curvesResult}.

\end{proof}

\section{Proof of Lemma \ref{mainLem}, Step 2: Proof of Lemma \ref{curvesResult}}
\begin{proof}[Proof of Lemma \ref{curvesResult}]
First, note that $I(\pts,\Gamma)\leq sn_2^2+\ell_2$. Arguing as above, we can assume $\ell_2<\rho n_2^{\frac{2+\epsilon}{d+1}}\ell_2^{\frac{d}{d+1}}$ or \eqref{IncidenceBdCurves} holds immediately; we can make $\rho>0$ smaller by making the constant $C^{\prime\prime}_\epsilon$ larger.

We will prove the statement by induction on $n_2+\ell_2$. Let $P$ be a partitioning polynomial of degree $D$, as given by Theorem \ref{th:partition}; there are $O_d(D^d)$ cells. By Theorem \ref{BBThm}, we have
\begin{equation}
|\{(\gamma,\Omega)\colon \gamma\in\Gamma,\ \Omega\ \textrm{a cell},\ \gamma\cap\Omega\neq\emptyset\}|=O_{d,C}(D \ell_2).
\end{equation}

Apply the induction hypothesis inside each cell. Either there is a point $p\in \pts_2\backslash\BZ(P)$ satisfying property (A.2), or by the same calculation as in \eqref{inductionCalc} we have
\begin{equation}\label{incidencesInsideCellsCalc}
I(\pts_2\backslash \BZ(P),\Gamma)\leq C^{\prime\prime}_\epsilon \big( n_2^{\frac{2+\epsilon}{d+1}}\ell_2^{\frac{d}{d+1}}+n_2\big).
\end{equation}

It remains to count incidences involving points on $\BZ(P)\cap\pts_2.$ Applying Lemma \ref{CoverVarietySmoothPts} to $\BZ_{\RR}(P)^*$, we can find a collection $\mathcal{W}$ of $O_D(1)$ complex algebraic varieties, and a partition $\{\pts_W\}_{W\in\mathcal{W}}$ of $\pts$ so that for each $W\in\mathcal{W},$ $\pts_W\subset W_{\reg}$.

By B\'ezout's theorem (recall, we are currently working with complex curves and varieties), for each $W\in\mathcal{W}$ we have
\begin{equation}\label{curveNotInVariety}
|\{(p,\alpha)\in\pts_W\times\Gamma_0\colon p^*\in\alpha,\ \alpha\not\subset W\}|=O_{C,D}(\ell_2).
\end{equation}

The total contribution from terms of the form \eqref{curveNotInVariety} summed over all $W\in\mathcal{W}$ is $O_{C,D}(\ell_2)$. If we select $C^{\prime\prime}_\epsilon$ large enough, we conclude that either
\begin{equation}\label{incidencesTransverseBoundary}
I(\pts_2\cap \BZ(P),\Gamma)\leq C^{\prime\prime}_\epsilon \ell_2,
\end{equation}
or there is a variety $W\in\mathcal{W}$ and a point $p\in \pts_W$ so that at least half the curves $\{\alpha\in\Gamma_0\colon p^*\in\alpha\}$ are contained in $W$. If the latter happens then
\begin{equation*}
|\{\alpha\in\Gamma_0\colon p^*\in\alpha,\ \dim(T_{p^*}\alpha\cap T_{p^*}(W))\geq 1\}|\geq\frac{1}{2}|\{\alpha\in\Gamma^*\colon p^*\in\alpha\},
\end{equation*}
where $T_{p^*}(W))$ is the (complex) Zariski tangent space of $W$. 

Let $\Pi\subset\RR^d$ be a $(d-1)$--flat containing $(T_{p^*}(W))(\RR)$. Then \eqref{manyCurveTangentToAPlane} holds with this choice of $p$ and $\Pi$. Thus, either (A.2) holds, or combining \eqref{incidencesInsideCellsCalc} and \eqref{incidencesTransverseBoundary} we obtain \eqref{IncidenceBdCurves}, i.e.~(B.2) holds.
\end{proof}

\bibliographystyle{abbrv}
\bibliography{richLinesBiblio}

\end{document}